\documentclass[11pt,reqno]{amsart}

\usepackage[T1]{fontenc}
\usepackage{lmodern}
\usepackage{microtype}
\usepackage{amsmath,amssymb,amsthm,mathtools}
\usepackage{enumitem}
\usepackage{xcolor}
\usepackage[colorlinks=true,linkcolor=blue!55!black,citecolor=green!40!black,urlcolor=blue!60!black]{hyperref}
\usepackage[nameinlink,capitalize,noabbrev]{cleveref}

\newtheorem{theorem}{Theorem}[section]
\newtheorem{lemma}[theorem]{Lemma}

\theoremstyle{definition}

\theoremstyle{remark}

\newcommand{\EE}{\mathbb E}
\newcommand{\PP}{\mathbb P}

\title[The toughness of random graphs]{The toughness of random graphs}

\author{Guang Li, Wenqian Zhang}
\address{School of Mathematics and Statistics, Shandong University of Technology, Zibo 255000, China}
\hypersetup{
  pdftitle={The toughness of random graphs},
  pdfauthor={Guang Li, Wenqian Zhang},
  pdfsubject={Random graphs and graph toughness},
  pdfkeywords={toughness, random graph, independence number, vertex cut}
}
\thanks{Supported by the National Natural Science Foundation of China (No.12301449, 12301448) and the Scientific Innovation Project for Young Scientists in Shandong Provincial Universities (No.2024KJG014).}
	\thanks{Corresponding author: Wenqian Zhang (zhangwq@pku.edu.cn)}

\begin{document}

\begin{abstract}
For a connected and non-complete  graph $G$ of order $n$, its toughness is defined as
\[
 \tau(G)=\min\bigl\{|S|/c(G-S):S\subseteq V(G),\ c(G-S)>1\bigr\},
\]
where $c(G-S)$ denotes the number of components of $G-S$. Let $\alpha(G)$ denote the independence number of $G$. An elementary bound on toughness is 
$$\tau(G)\leq\frac{n-\alpha(G)}{\alpha(G)}.$$
 Fix $p\in(0,1)$, and let $G(n,p)$ be the binomial random graph on vertex set $[n]$. Set $a=\alpha(G(n,p))$. In this paper, we mainly prove that
\[
 \tau(G(n,p))\in\left\{\frac{n-a}{a},\frac{n-a-1}{a}\right\}
\] 
with high probability.

\end{abstract}

\maketitle

\section{Introduction}

All graphs considered in this paper are finite and simple.  We use $\alpha(G)$ for the independence number, $\kappa(G)$ for vertex connectivity, and $c(G)$ for the number of components of $G$. The toughness of graphs was introduced by Chv\'atal \cite{Chvatal1973} as a global measure of resistance to vertex deletion.  If $G$ is a  connected and non-complete graph, then its toughness is defined as
\begin{equation}\label{eq:toughness-definition}
 \tau(G)=\min_{\substack{S\subseteq V(G)\\c(G-S)>1}}
 \frac{|S|}{c(G-S)}.
\end{equation}
  This parameter is closely related with Hamilton cycles, factors, expansion, and separator structure of graphs (see the survey paper \cite{BBS2006}).  It is usually difficult to determine the exact value of toughness of graphs (see \cite{BauerHakimiSchmeichel1990}). The toughness of some special families of graphs was determined (for example,  strongly regular graphs \cite{CioabaWong2014} and substantial ranges of Kneser graphs \cite{ParkEtAl2021}).  

In this paper, we study  the toughness of random graphs.  Let $G(n,p)$ be the binomial random graph on vertex set $[n]$, with every edge present independently with probability $p$.  An event holds \emph{with high probability}, abbreviated whp, if its probability tends to one as $n\to\infty$. Throughout, $p\in(0,1)$ is fixed, and set $q=1-p$ and
$b=\frac1q.$
The classical theory of random graphs gives
\begin{equation}\label{eq:standard-random-parameters}
 \alpha(G(n,p))=(2+o(1))\log_b n,
 \qquad
 \kappa(G(n,p))=(p+o(1))n
\end{equation}
with high probability.  The first estimate goes back to the classical clique-number results of Matula and of Bollob\'as and Erd\H{o}s (see \cite{Matula1970,BollobasErdos1976}).  The second one follows from the random-graph connectivity theory of Erd\H{o}s--R\'enyi and Bollob\'as--Thomason (see \cite{ErdosRenyi1961,BollobasThomason1985};  also \cite[Chapters 7 and 11]{Bollobas2001}).  Recently, the independence number of random graphs has received considerable attention (see \cite{BohmanHofstad2024}).

For a non-complete and connected graph $G$ of order $n$, the complement of a maximum independent set gives the elementary upper bound
\begin{equation}\label{eq:independence-upper-intro}
 \tau(G)\le \frac{n-\alpha(G)}{\alpha(G)}.
\end{equation}
It is therefore natural to ask whether this construction describes the toughness of a graph, either asymptotically or exactly.  Our first result  affirmatively answers the asymptotic question.

\begin{theorem}\label{thm:main-asymptotic}
Given $p\in(0,1)$, set $b=(1-p)^{-1}$ and $a=\alpha(G(n,p))$.  Then, with high probability,
\begin{equation}\label{eq:quantitative-envelope}
\tau(G(n,p))\in\left\{\frac{n-a}{a},\frac{n-a-1}{a}\right\}.
\notag
\end{equation}
Consequently,
\[
 \tau(G(n,p))
 =(1+o(1))\frac{n}{2\log_b n}
\] with high probability.
\end{theorem}

One may ask whether the probability that
$\tau(G(n,p))=(n-\alpha(G(n,p)))/\alpha(G(n,p))$ is bounded below by a
positive constant for all sufficiently large $n$. Our next result answers
this question negatively by exhibiting a sequence along which this
probability tends to zero. (It makes no assertion that this probability
tends to zero along the full sequence of integers $n$.)

\begin{theorem}\label{thm:exact-instability}
Fix $p\in(0,1)$.  There is a sequence of integers $n_m\to\infty$ (for $m\to\infty$) such that, for $G\sim G(n_m,p)$, with high probability
\begin{enumerate}[label=\textup{(\roman*)}]
 \item $\alpha(G)=m-1$;
 \item $
 \tau(G)\le \frac{n_m-\alpha(G)-1}{\alpha(G)}.
$
\end{enumerate}
\end{theorem}

The rest of this paper is organized as follows. In Section 2, we give some elementary estimates. 
 In Section 3, we will prove a useful lemma for \cref{thm:main-asymptotic}. The proof of \cref{thm:main-asymptotic} will be given in Section 4, and the proof of \cref{thm:exact-instability} will be given in Section 5.

\section{Some elementary estimates}

  We shall repeatedly use the following elementary estimates.  

\begin{lemma}\label[lemma]{lem:elementary-estimates}
Let $1\leq r \leq N$ and $X\sim\operatorname{Bin}(N,s)$. Then the following conclusions hold.
\begin{enumerate}[label=\textup{(\roman*)}]
 \item 
 \begin{equation}\label{eq:binomial-coefficient-upper}
  \binom Nr\le\left(\frac{eN}{r}\right)^r.
 \end{equation}
 \item For any integer
 $z\ge\max\{1,Ns\}$,
 \begin{equation}\label{eq:binomial-tail-self-contained}
  \PP(X\ge z)\le\left(\frac{eNs}{z}\right)^z.
 \end{equation}
 \item If $r/N\to0$ (for $N\to\infty$), then
 \begin{equation}\label{eq:uniform-stirling-binomial}
  \log\binom Nr
  =r\log\frac Nr+r
   +O\!\left(\frac{r^2}{N}+\log(r+1)\right).
 \end{equation}
\end{enumerate}
\end{lemma}

\begin{proof}
For (i), using
\[
 \log(r!)=\sum_{i=1}^r\log i
 \ge\int_1^r\log x\,dx
 =r\log r-r+1,
\]
we obtain $r!\ge e(r/e)^r\ge(r/e)^r$.  Since
\[
 \binom Nr
 =\frac{N(N-1)\cdots(N-r+1)}{r!}
 \le\frac{N^r}{r!},
\]
we see
\[
 \binom Nr\le\frac{N^r}{(r/e)^r}
 =\left(\frac{eN}{r}\right)^r,
\]
establishing \eqref{eq:binomial-coefficient-upper}.

For (ii), if $z>N$, the event $X\ge z$ is impossible and the inequality
is immediate. We may therefore assume $1\le z\le N$.
Set $Y=\binom Xz$ as $X$ is integer-valued. Clearly, the event $X\ge z$ is equivalent to the event $Y\ge1$.  Using Markov's
inequality we obtain
\begin{equation}\label{eq:factorial-moment-tail}
 \PP(X\ge z)=\PP(Y\ge 1)\le\EE Y
 =\EE\binom Xz.
\end{equation}
Denote $X=\sum_{i=1}^N\xi_i$, where the $\xi_i$ are independent Bernoulli
variables with success probability $s$.  Since the random variable $\binom Xz$
counts the $z$-subsets of successful trials, we have
\[
 \binom Xz
 =\sum_{J\in\binom{[N]}z}\prod_{i\in J}\xi_i.
\]
Then
\[
 \EE\binom Xz
 =\sum_{J\in\binom{[N]}z}s^z
 =\binom Nzs^z.
\]
Using \eqref{eq:factorial-moment-tail} and (i) with $r=z$, we obtain 
\[
 \PP(X\ge z)
 \le\binom Nzs^z
 \le\left(\frac{eN}{z}\right)^zs^z
 =\left(\frac{eNs}{z}\right)^z,
\]
establishing (ii).  (In fact, the calculation before the final numerical
comparison is valid for every integer $1\le z\le N$; the stated assumption
$z\ge Ns$ is the range in which the displayed bound is useful.)

For (iii), we can assume $r/N\le1/2$ as $r/N\to0$.  It is easy to check that
\begin{equation}\label{eq:binomial-product-expansion}
 \log\binom Nr
 =r\log N-\log(r!)
  +\sum_{i=0}^{r-1}\log\left(1-\frac{i}{N}\right).
\end{equation}
Since $\log x$ is increasing for $x>0$, we have
\[
 \int_1^r\log x\,dx
 \le\sum_{i=2}^r\log i
 \le\int_1^r\log x\,dx+\log r.
\]
It is easy to see 
$$\int_1^r\log x\,dx=r\log r-r+1.$$
Thus, the two-sided estimate above is
\[
 \log(r!)=r\log r-r+O(\log(r+1)),
\]
where $\log(r+1)$ also absorbs the constant term when $r=1$.
For $0\le x\le1/2$, using Taylor's formula we have 
$\log(1-x)=-x+O(x^2)$.  
Hence
\begin{align*}
 \sum_{i=0}^{r-1}\log\left(1-\frac{i}{N}\right)
 &=-\frac1N\sum_{i=0}^{r-1}i
   +O\!\left(\frac1{N^2}\sum_{i=0}^{r-1}i^2\right)\\
 &=-\frac{r(r-1)}{2N}+O\!\left(\frac{r^3}{N^2}\right)
 =O\!\left(\frac{r^2}{N}\right),
\end{align*}
where the last equality uses $r/N\le1/2$.  Inserting the last two estimates
into \eqref{eq:binomial-product-expansion} gives
\[
 \log\binom Nr
 =r\log N-r\log r+r
  +O\!\left(\frac{r^2}{N}+\log(r+1)\right),
\]
establishing (iii).
\end{proof}

For $x>0$, let $\log^+x:=\max\{\log x,0\}$. For an integer
$t\geq0$, set
\[
 H_j^{(t)}=
 \frac{\binom tj\binom{n-t}{t-j}}{\binom nt},\qquad
 T_j^{(t)}=H_j^{(t)}b^{\binom j2}.
\]

\begin{lemma}\label[lemma]{lem:overlap-sum}
Let $t$ be an integer satisfying that
$t=(2+o(1))\log_b n $
and
 $\log^+(\mu_t^{-1})=o(\log n),
$ where $\mu_t=\binom ntq^{\binom t2}$.
Then
\[
 \sum_{j=2}^{t-1}T_j^{(t)}=o(1).
\]
\end{lemma}

\begin{proof} Fix a $t$-set $A$  and choose a uniformly random $t$-set $B$.  Then
$H_j^{(t)}=\PP(|A\cap B|=j)$.  If $J\subseteq A$ has order $j$, then
\[
 \PP(J\subseteq B)=\frac{\binom{n-j}{t-j}}{\binom{n}{t}}
 \le\left(\frac{t}{n-t}\right)^j.
\]
Note that $n-t\ge n/2$ for all sufficiently
large $n$, since $t=(2+o(1))\log_b n=o(n)$.   Hence, for $2\le j\le t-1$,
\begin{equation}\label{eq:small-overlap-hypergeometric}
 H_j^{(t)}
 \le \PP(|A\cap B|\ge j)
 \le \binom tj\left(\frac{t}{n-t}\right)^j
 \le\left(\frac{2et^2}{jn}\right)^j.
\end{equation}
Consequently, for $2\le j\le t/2$,
\[
 T_j^{(t)}\le
 \left(\frac{Ct^2b^{(j-1)/2}}{jn}\right)^j.
\]
By  the assumption on $t$, we see 
\begin{equation}\label{eq:bt-overlap-scale}
 b^t=n^{2+o(1)},\qquad t=n^{o(1)}.
\end{equation}
If $2\le j\le t/4$, then
$b^{(j-1)/2}\le b^{t/8}=n^{1/4+o(1)}$. Thus, $$T_j^{(t)}\leq(n^{-3/4+o(1)})^{j}\leq(n^{-3/4+o(1)})^{2}.$$
  Therefore
\begin{equation}\label{eq:small-overlap-sum}
 \sum_{j=2}^{\lfloor t/4\rfloor}T_j^{(t)}
 \le t\bigl(n^{-3/4+o(1)}\bigr)^2=o(1).
\end{equation}
If $t/4<j\le t/2$, then
$b^{(j-1)/2}\le b^{t/4}=n^{1/2+o(1)}$ and $t^2/j=O(t)=n^{o(1)}$.
Thus, $$T_j^{(t)}\leq(n^{-1/2+o(1)})^{j}\leq(n^{-1/2+o(1)})^{t/4}.$$
 Since $j>t/4$ and $t=\Theta(\log n)$,
\begin{equation}\label{eq:middle-overlap-sum}
 \sum_{t/4<j\le t/2}T_j^{(t)}
 \le t\bigl(n^{-1/2+o(1)}\bigr)^{t/4}=o(1).
\end{equation}

It remains to consider $t/2<j<t$.  Set $j=t-s$, where
$1\le s<t/2$.  From the definitions of $H_j^{(t)}$, $T_j^{(t)}$, and
$\mu_t=\binom ntb^{-\binom t2}$, we see
\[
 T_{t-s}^{(t)}
 =\mu_t^{-1}\binom ts\binom{n-t}s
 b^{-s(2t-s-1)/2}
 .
\]
Using $\binom ts\le(et/s)^s$ and
$\binom{n-t}s\le(en/s)^s$, we obtain
\begin{align*}
 T_{t-s}^{(t)}
 &\le \mu_t^{-1}
 \left(\frac{e^2tn}{s^2}
 b^{-t+(s+1)/2}\right)^s.
\end{align*}
Since $s<t/2$, we have
$-t+(s+1)/2\le-3t/4+1/2$.  By
\eqref{eq:bt-overlap-scale},
\[
 n b^{-t+(s+1)/2}
 \le n^{-1/2+o(1)}.
\]
The assumption $\log^+(\mu_t^{-1})=o(\log n)$ is equivalent to
$\mu_t^{-1}\le n^{\varepsilon_n}$ for some 
$\varepsilon_n\to0$.  Consequently, uniformly for $1\le s<t/2$,
\begin{equation}\label{eq:large-overlap-term}
 T_{t-s}^{(t)}
 \le n^{\varepsilon_n}
 \left(\frac{C_bt}{s^2}n^{-1/2+o(1)}\right)^s.
\end{equation}
Noting $t=O(\log n)$, the expression in parentheses is at most
$n^{-1/3}$ for all sufficiently large $n$.
By taking $\varepsilon_n<1/6$, from
\eqref{eq:large-overlap-term} we obtain
\begin{equation}\label{eq:large-overlap-sum}
 \sum_{1\le s<t/2}T_{t-s}^{(t)}
 \le n^{\varepsilon_n}\sum_{s\ge1}n^{-s/3}
 =o(1).
\end{equation}
Now the lemma follows from \eqref{eq:small-overlap-sum},
\eqref{eq:middle-overlap-sum}, and \eqref{eq:large-overlap-sum}.
\end{proof}

\begin{lemma}\label[lemma]{lem:connected-noncomplete}
Given $p\in(0,1)$, then $G(n,p)$ is connected and non-complete with
probability $1-o(1)$.
\end{lemma}

\begin{proof}
Set $q=1-p$ and $c=-\log q>0$.  If a graph is disconnected, it has a
 component with vertex set $U$ satisfying
$1\le|U|\le n/2$.  The event that 
the $k(n-k)$ edges between $U$ and its complement are absent  has probability
$q^{k(n-k)}$.  Thus, using $n-k\ge n/2$ and
\cref{lem:elementary-estimates}(i),
\begin{align*}
 \PP(G(n,p)\text{ is disconnected})
 &\le\sum_{k=1}^{\lfloor n/2\rfloor}
   \binom nkq^{k(n-k)}\\
 &\le\sum_{k=1}^{\lfloor n/2\rfloor}
   \left(\frac{en}{k}\right)^k e^{-ckn/2}.
\end{align*}
Clearly,
$en e^{-cn/2}\le e^{-cn/3}$ for all sufficiently large $n$.  Hence
\[
 \PP(G(n,p)\text{ is disconnected})
 \le\sum_{k\ge1}e^{-ckn/3}
 =\frac{e^{-cn/3}}{1-e^{-cn/3}}=o(1).
\]
Noting $\log p\le-(1-p)$, the probability that $G(n,p)$ is complete equals
\[
 p^{\binom n2}\le e^{-(1-p)\binom n2}=o(1).
\]
Consequently, $G(n,p)$ is connected and non-complete with probability
$1-o(1)$.
\end{proof}

\section{A useful lemma}

Let $G$ be a graph with vertex set $[n]$. For $U\subseteq[n]$, define
\[
 B(U)=\{x\in[n]\setminus U:|N(x)\cap U|\le1\},
\]
where $N(x)$ denotes the set of neighbors of $x$ in $G$.
The following lemma is used in the proof of \cref{thm:main-asymptotic}.

\begin{lemma}\label[lemma]{lem:sparse-interface}
Fix $p\in(0,1)$ and constants $0<\theta<C$.  Let $b=(1-p)^{-1}$.  There is a constant $\gamma=\gamma(p,\theta,C)>0$ such that, whp $G(n,p)$ has the property:
\[
 |B(U)|\le n^{1-\gamma}=o(n),
\]
for any $U\subseteq[n]$ with $
 \theta\log_b n\le |U|\le C\log_b n$.
\end{lemma}

\begin{proof}
Set $q=1-p$ and  $t=|U|$.  For
$x\in[n]\setminus U$, the random variable $|N(x)\cap U|$ has distribution
$\operatorname{Bin}(t,p)$.  Therefore
\[
 \pi_t:=\PP(x\in B(U))
 =q^t+tpq^{t-1}
 =\left(1+\frac{pt}{q}\right)q^t.
\]
For each $x\notin U$, the event $x\in B(U)$ is determined solely by the
edges $\{\{x,u\}:u\in U\}$. These edge sets are pairwise disjoint as $x$
varies over $[n]\setminus U$. Since all edge indicators in $G(n,p)$ are
mutually independent, the events $\{x\in B(U)\}$ are mutually independent.
Thus,
\[
 |B(U)|\sim\operatorname{Bin}(n-t,\pi_t).
\]
 Note that
$q^t=b^{-t}\le n^{-\theta}$ as $t\ge\theta\log_b n$.  Noting that $t\le C\log_b n$, there is a
constant $C_1=C_1(p,C)>0$ such that
$1+pt/q\le C_1\log n$ for all sufficiently large $n$.  
Thus,
\begin{equation}\label{eq:BU-mean}
 (n-t)\pi_t\le C_1n^{1-\theta}\log n
 =n^{1-\theta+o(1)}.
\end{equation}

Choose $0<\gamma<\min\{\theta/3,1/3\}$ and set
$z=\lceil n^{1-\gamma}\rceil$.  Then $z>(n-t)\pi_t$, and 
$z<n-t$ as $t=O(\log n)$. 
Noting $\theta-\gamma>0$,
\eqref{eq:BU-mean} implies
\[
 \frac{e(n-t)\pi_t}{z}
 \le eC_1n^{-(\theta-\gamma)}\log n
 \le n^{-(\theta-\gamma)/2}
\]
for all sufficiently large $n$.    We can therefore apply
\cref{lem:elementary-estimates}(ii) to
$|B(U)|\sim\operatorname{Bin}(n-t,\pi_t)$.  It gives
\[
 \PP(|B(U)|\ge z)
 \le\left(\frac{e(n-t)\pi_t}{z}\right)^z
\le \exp\left\{-\frac{\theta-\gamma}{2}z\log n\right\}
 \le \exp\bigl(-c_2n^{1-\gamma}\log n\bigr)
\]
for a constant $c_2=c_2(\theta,\gamma)>0$.

The number of such sets $U$ is at most
\[
 \sum_{t\le C\log_b n}\binom nt
 \le(C\log_b n+1)n^{C\log_b n}
 =\exp\bigl(O((\log n)^2)\bigr).
\]
Since $n^{1-\gamma}\log n\gg(\log n)^2$, all such sets $U$ satisfies $|B(U)|<z$ with probability $1-o(1)$.
This completes the proof.
\end{proof}

\section{Proof of \texorpdfstring{\cref{thm:main-asymptotic}}{Theorem 1.1}}

In this section, we give the proof \cref{thm:main-asymptotic}.
Recall that $q=1-p$ and $b=q^{-1}$.
Define
\begin{equation}\label{eq:mu-lambda}
 \mu_s=\binom ns q^{\binom s2},
\end{equation}
and 
\begin{equation}\label{eq:canonical-r}
 r=r(n)=\max\{1\le s\le n:\mu_s\ge1\}.
\end{equation}

\begin{lemma}\label[lemma]{lem:canonical-scale}
Let $r=r(n)$ be defined in \eqref{eq:canonical-r}. Then
\begin{equation}\label{eq:scale-bounds}
 \frac{r-1}{2}\log b-1
 \le \log\frac nr
 \le \frac r2\log b+o(1).
\end{equation}
Consequently,
\begin{equation}\label{eq:br-scale}
 b^r=\Theta\!\left(\frac{n^2}{r^2}\right),
 \end{equation}
and
\begin{equation}\label{eq:ratio-nu-mu}
r=2\log_b n-2\log_b\log_b n+O(1).
\end{equation}
\end{lemma}

\begin{proof}
First, for every fixed positive integer $s$,
$\mu_s(n)=\binom nsq^{\binom s2}\to\infty$ as $n\to\infty$.
The definition of $r$ therefore implies $r(n)\to\infty$.
Also $\mu_n=q^{\binom n2}<1$ for $n\ge2$, so $r<n$ and
$\mu_{r+1}<1$ is well defined.
Since $\mu_r\ge1$, we have
\[
 1\le\binom nr b^{-\binom r2}
 \le\left(\frac{en}{r}\right)^r b^{-r(r-1)/2},
\]
which gives the lower bound in \eqref{eq:scale-bounds}. In particular,
$(r-1)\log b/2\le\log n+1$, so $r=O(\log n)$ and $r/n\to0$.
Since
$\mu_{r+1}<1$ and
\[
 \binom n{r+1}\ge
 \left(\frac{n-r}{r+1}\right)^{r+1},
\]
we obtain
\[
 \log\frac{n-r}{r+1}<\frac r2\log b.
\]
 Since $r\to\infty$ and $r/n\to0$,
\[
 \log\frac{n-r}{r+1}-\log\frac nr
 =\log(1-r/n)-\log(1+1/r)=o(1).
\]
This proves the upper bound. Exponentiation gives
$b^r=\Theta((n/r)^2)$, and then
$r=2\log_b(n/r)+O(1)$ gives the displayed expansion for $r$.
\end{proof}

For a graph $G$, let $a=\alpha(G)$. Define 
\begin{equation}\label{eq:rho-definition}
 \rho_a(G)=\max\{|W|:W\subseteq V(G),\ c(G[W])=a\}.
\end{equation}

\begin{lemma}\label[lemma]{lem:endpoint-forest}
Set $a=\alpha(G(n,p))$ and $\rho_a=\rho_a(G(n,p))$. With high probability, $r-1\leq a\leq r+1$ and $a\leq\rho_a\leq a+1$.
\end{lemma}

\begin{proof} We first show that  $r-1\leq a\leq r+1$ with high probability. 
Using \eqref{eq:br-scale} of \cref{lem:canonical-scale}, we can obtain
\begin{equation}\label{eq:neighboring-means}
 \frac{\mu_{r-1}}{\mu_r}=\Theta(n/r),\qquad
 \frac{\mu_{r+2}}{\mu_{r+1}}=\Theta(r/n).
\end{equation}
Hence
\begin{equation}\label{eq:neighboring-consequences}
 \mu_{r-1}\to\infty,\qquad \mu_{r+2}\to0.
\end{equation}

Let $X_{r+2}$ be the number of  independent sets of order $r+2$ in $G(n,p)$. Then 
$\EE X_{r+2}=\mu_{r+2}\to0$. Using Markov's inequality we have 
$\mathbb{P} (X_{r+2}\geq1)\leq\EE X_{r+2}\rightarrow0$. It follows that $a\leq r+1$ with high probability.

Let $X_{r-1}=\sum_A\xi_A$, where the sum is over all $r-1$-subsets of $[n]$ and
$\xi_A$ is the indicator that $A$ is independent.  Thus
$\EE X_{r-1}=\mu_{r-1}$.  Clearly, if two distinct $r-1$-sets $A,B$ have at most one common
vertex,  the covariance of $\xi_A$ and $\xi_B$  will be zero.  If
$|A\cap B|=j\ge2$, then
\[
 \EE(\xi_A\xi_B)
 =q^{2\binom{r-1}{2}-\binom j2}
 =q^{2\binom{r-1}{2}}b^{\binom j2}.
\]
Write the variance of the indicator sum as
\[
 \operatorname{Var}X_{r-1}
 =\sum_A\operatorname{Var}\xi_A
  +\sum_{\substack{A,B\in\binom{[n]}{r-1}\\A\ne B}}
    \operatorname{Cov}(\xi_A,\xi_B),
\]
where the second sum is over ordered pairs.  Since $\xi_A^2=\xi_A$, we have $ \operatorname{Var}\xi_A\leq\EE\xi_A$. Thus, using $\operatorname{Cov}(\xi_A,\xi_B)\leq\mathbb{E}(\xi_A\xi_B)$,
\begin{align*}
 \operatorname{Var}X_{r-1}
 &\le \sum_A\EE\xi_A
 +\sum_{j=2}^{r-2}
   \binom{n}{r-1}\binom{r-1}{j}\binom{n-r+1}{r-1-j}
   q^{2\binom{r-1}{2}} b^{\binom j2}.
\end{align*}
Since $\sum_A\EE\xi_A=\mu_{r-1}$ and
$\mu_{r-1}^2=\binom{n}{r-1}^2q^{2\binom{r-1}{2}}$, we have
\[
 \frac{\operatorname{Var}X_{r-1}}{\mu_{r-1}^2}
 \le\frac1\mu_{r-1}+
 \sum_{j=2}^{r-2}
 \frac{\binom{r-1}{j}\binom{n-r+1}{r-1-j}}{\binom{n}{r-1}}
 b^{\binom j2}
 =\frac1\mu_{r-1}+\sum_{j=2}^{r-2}T_j^{(r-1)}.
\]
The first term tends to zero because $\mu_{r-1}\to\infty$, and the sum is
$o(1)$ by \cref{lem:overlap-sum}.  Hence
\[
 \frac{\operatorname{Var}X_{r-1}}{\mu_{r-1}^2}=o(1).
\]
Using Chebyshev's inequality, we have
\[
 \PP\left(X_{r-1}=0\right)\leq
 \PP\left(\left|X_{r-1}-\mu_{r-1}\right|\ge\mu_{r-1}\right)
 \le \frac{\operatorname{Var}X_{r-1}}{\mu_{r-1}^2}=o(1).
\]
It follows that
\[
 \PP(X_{r-1}>0)=1-o(1).
\]
Therefore, $a\geq r-1$ with high probability.
Consequently, $r-1\leq a\leq r+1$ with high probability.

We next prove $a\leq\rho_a\leq a+1$ with high probability.  The lower
bound is deterministic: a maximum independent set has $a$ vertices and its
induced graph has exactly $a$ components. We shall prove the upper bound in the following.

For each deterministic integer $t$, let $Z_t$ be the number of
$(t+2)$-subsets $W\subseteq[n]$ for which $c(G[W])=t$.  We first record why
these variables detect the event $\rho_a\ge a+2$.  Suppose that a set
$W_0$ satisfies $c(G[W_0])=a$ and $|W_0|\ge a+2$.  
If two components have order at least two, choose the endpoints of one edge
from each of them and one vertex from every other component.  If instead one
component has order at least three, choose three vertices that induce a
connected graph in that component: take three consecutive vertices of a
shortest path between two vertices at distance at least two; if the component
is complete, any three vertices work.  Also choose one vertex from every other
component.  In either case the selected vertices form a set $W\subseteq W_0$
with
\[
 |W|=a+2,\qquad c(G[W])=a.
\]
Consequently,
\begin{equation}\label{eq:rho-detected-by-Z}
 \{\rho_a\ge a+2\}\subseteq\{Z_a>0\}.
\end{equation}

We now estimate $Z_t$ without treating the random variable $a$ as a
deterministic parameter.  A graph on $t+2$ vertices with exactly $t$
components has one of only two component-order profiles:
\[
 3,1,\ldots,1
 \quad\hbox{or}\quad
 2,2,1,\ldots,1.
\]
For a fixed $(t+2)$-set, the vertices in the nontrivial components can be
chosen in at most $C(t+2)^4$ ways.  Once they are chosen, all edges between
distinct components must be absent.  There are
$\binom{t+2}{2}-3$ such forced nonedges in the first profile and
$\binom{t+2}{2}-2$ in the second.  The probabilities of the required
internal edges are at most one.  Since $p$ and $q$ are fixed, a union bound
therefore gives, for a constant $C_p$ depending only on $p$,
\begin{equation}\label{eq:Zt-first-moment}
 \EE Z_t\le
 C_p t^4\binom n{t+2}q^{\binom{t+2}{2}}.
\end{equation}

We apply this estimate for the three possible values of $a$.  First, since
$\mu_{r+1}<1$, \eqref{eq:Zt-first-moment}, the exact ratio of consecutive
binomial coefficients, and \eqref{eq:br-scale} yield
\begin{align}
 \EE Z_{r+1}
 &\le C_pr^4\mu_{r+1}
   \frac{\binom n{r+3}}{\binom n{r+1}}q^{2r+3}\notag\\
 &\le C'_pr^4\frac{n^2}{r^2}b^{-2r}
 =O\!\left(\frac{r^6}{n^2}\right)=o(1).
 \label{eq:Z-r-plus-one}
\end{align}
Similarly,
\begin{align}
 \EE Z_r
 &\le C_pr^4\mu_{r+1}
   \frac{\binom n{r+2}}{\binom n{r+1}}q^{r+1}\notag\\
 &\le C'_pr^4\frac nr b^{-r}
 =O\!\left(\frac{r^5}{n}\right)=o(1).
 \label{eq:Z-r}
\end{align}
Here and below we use $r=O(\log n)$, which follows from
\cref{lem:canonical-scale}.  Markov's inequality shows that
$Z_{r+1}=Z_r=0$ with high probability.

The value $a=r-1$ requires a short dichotomy because $\mu_r$ need not have
a limit along the full sequence of integers $n$.  Along any subsequence on
which $\mu_r$ is bounded by a constant, \eqref{eq:Zt-first-moment} gives
\begin{align}
 \EE Z_{r-1}
 &\le C_pr^4\mu_r
   \frac{\binom n{r+1}}{\binom nr}q^r\notag\\
 &\le C'_pr^4\frac nr b^{-r}
 =O\!\left(\frac{r^5}{n}\right)=o(1).
 \label{eq:Z-r-minus-one}
\end{align}
Thus $Z_{r-1}=0$ with high probability on every such subsequence.

Along any subsequence on which $\mu_r\to\infty$, let
$X_r=\sum_A\zeta_A$, where the sum is over all $r$-subsets of $[n]$ and
$\zeta_A$ indicates that $A$ is independent.  Repeating the covariance
calculation already given for $X_{r-1}$ gives
\[
 \frac{\operatorname{Var}X_r}{\mu_r^2}
 \le\frac1{\mu_r}+
 \sum_{j=2}^{r-1}
 \frac{\binom rj\binom{n-r}{r-j}}{\binom nr}
 b^{\binom j2}
 =\frac1{\mu_r}+\sum_{j=2}^{r-1}T_j^{(r)}.
\]
The hypotheses of \cref{lem:overlap-sum} hold: by
\cref{lem:canonical-scale}, $r=(2+o(1))\log_b n$, and
$\log^+(\mu_r^{-1})=0$ because $\mu_r\ge1$.  Hence the overlap sum is
$o(1)$, while $1/\mu_r=o(1)$ on the present subsequence.  Chebyshev's
inequality now gives
\[
 \PP(X_r=0)
 \le\PP\bigl(|X_r-\mu_r|\ge\mu_r\bigr)
 \le\frac{\operatorname{Var}X_r}{\mu_r^2}=o(1).
\]
Since $a=r-1$ implies $X_r=0$, we obtain
\begin{equation}\label{eq:a-r-minus-one-unlikely}
 \PP(a=r-1)\le\PP(X_r=0)=o(1)
\end{equation}
on every subsequence on which $\mu_r\to\infty$.

For completeness, these two cases cover a possibly oscillating sequence
$\mu_r$.  Indeed, every subsequence has a further subsequence on which
$\mu_r$ is bounded, or else a further subsequence on which
$\mu_r\to\infty$.  On the first type, \eqref{eq:Z-r-minus-one} rules out
$Z_{r-1}>0$; on the second type, \eqref{eq:a-r-minus-one-unlikely} rules
out $a=r-1$.  It follows by the subsequence criterion for convergence that
\[
 \PP(a=r-1\ \hbox{and}\ Z_{r-1}>0)=o(1)
\]
along the full sequence.  Combining this with
\eqref{eq:rho-detected-by-Z}, \eqref{eq:Z-r-plus-one},
\eqref{eq:Z-r}, and the already proved localization
$a\in\{r-1,r,r+1\}$ shows that
$\PP(\rho_a\ge a+2)=o(1)$.  Therefore
$a\le\rho_a\le a+1$ with high probability.
\end{proof}

\begin{proof}[Proof of \cref{thm:main-asymptotic}]
 Set $k=\kappa(G(n,p))$. By \eqref{eq:standard-random-parameters}, with probability $1-o(1)$, 
\[
 \log_b n\le a\le 3\log_b n,
 \qquad k=(p+o(1))n.
\]
Using \cref{lem:sparse-interface} (with
$\theta=p/3$ and $C=4$) and 
\cref{lem:endpoint-forest}, we see that, with probability $1-o(1)$, $G(n,p)$ has the following property:\\
(a1) for any $U\subseteq[n]$ with $
 \theta\log_b n\le |U|\le C\log_b n$. \[
 |B(U)|\le n^{1-\gamma}=o(n),
\]
where $\gamma>0$ is independent of $n$;\\
 (a2) 
\[
 a\leq \rho_{a}\leq a+1,
\]
where $\rho_{a}=\rho_{a}(G(n,p))$.
 
We restrict $G(n,p)$ on the above events with probability $1-o(1)$, and denote $G=G(n,p)$. Clearly, $G$ is connected and non-complete, since $k>0$
and $a\ge2$.  Its toughness is well defined.  
Let $S$ be a cut set of $G$ such that $\tau(G)=\frac{|S|}{c(G-S)}$.  Set
\[
 W=V(G)\setminus S,\qquad w=|W|,
 \qquad c=c(G[W]).
\]

Now we show that $pa/2<c\le a$.  Choosing one vertex from each component
of $G[W]$ produces an independent set of order $c$, so $c\le\alpha(G)=a$.
Suppose, for a contradiction, that $c\le pa/2$.
Then, noting $|S|\ge k$,
\[
 \tau(G)=\frac{|S|}{c}
 \ge\frac{k}{c}
 \ge\frac{k}{pa/2}
 =(2-o(1))\frac na>\frac{n}{a}.
\]
For all sufficiently large $n$, the last expression is larger than $n/a$.
This is impossible, since $\tau(G)\leq(n-a)/a<n/a$ by the elementary
estimate. Thus, $pa/2<c\le a$.

Recall $pa/2<c\le a$.
Let $U$ be an independent set obtained by choosing one representative from each component of $G[W]$.  Then
$|U|=c$.  Note that
\[
 \theta\log_b n\le c\le C\log_b n
\] as  $c>pa/2$.
Thus by (a1), noting $W\subseteq U\cup B(U)$,
\begin{equation}\label{eq:coarse-r}
 w=|W|\le |U|+|B(U)|\le c+n^{1-\gamma}.
\end{equation}

Suppose first that $c\le a-1$. Recall $a=O(\log n)$. By
\eqref{eq:coarse-r}, for all sufficiently large $n$,
\[
 a(w-c)\le a n^{1-\gamma}
 =O(\log n)n^{1-\gamma}<n\le(a-c)n.
\]
Rearranging the strict inequality gives
$a(n-w)>c(n-a)$, and division by $ac>0$ gives
\begin{equation}\label{eq:strict-gap}
 \tau(G)=\frac{n-w}{c}>\frac{n-a}{a},
\end{equation}
a contradiction. 

It remains to consider $c=a$.  Now $U$ is a maximum independent set of
$G$, and $\tau(G)=(n-w)/a$.  Since $c(G[W])=a$, the definition of
$\rho_a$ gives $w\le\rho_a$.  Conversely, choose a set $W^*$ of order
$\rho_a$ with $c(G[W^*])=a$.  Its complement is an admissible cut set, so
the minimality in the definition of toughness gives
\[
 \tau(G)\le\frac{n-\rho_a}{a}\le\frac{n-w}{a}=\tau(G).
\]
Both inequalities are therefore equalities.  In particular, $w=\rho_a$.
Using (a2), we conclude that
\[
 \tau(G)=\frac{n-\rho_a}{a}
 \in\left\{\frac{n-a}{a},\frac{n-a-1}{a}\right\}.
\]
This completes the proof.
\end{proof}

\section{Proof of \texorpdfstring{\cref{thm:exact-instability}}{Theorem 1.2}}

In this section, we give the proof of \cref{thm:exact-instability}.  
For integers $n\ge m$, let
\[
 \mu_m(n)=\binom nm q^{\binom m2},
 \qquad q=1-p.
\]
We first show the following lemma.

\begin{lemma}\label[lemma]{lem:one-edge-second-moment}
Fix $p\in(0,1)$.  Suppose $m\to\infty$, $n=n(m)$, and
\begin{equation}\label{eq:critical-mu}
 \mu_m(n)\sim\frac1m.
\end{equation}
If $Y_m$ denotes the number of $m$-vertex sets inducing exactly one edge in $G(n,p)$, then
$$\PP(Y_m>0)\to1.$$
\end{lemma}

\begin{proof}
Set $N=\binom m2$ and $b=q^{-1}$. We first show that $n/m\to\infty$.  Indeed, if
 $n\le Km$ for some constant
$K>1$, then
\[
 \mu_m(n)\le\left(\frac{en}{m}\right)^m q^{\binom m2}
 \le(eK)^m q^{\binom m2}
 =\exp\{-\tfrac12m^2\log b+O(m)\}=o(1/m),
\]
a contradiction to $\mu_m(n)\sim1/m$. Thus, $n/m\to\infty$ or $m/n\to0$.

Fix an $m$-set $A$.  Its induced graph has $N$ possible edges.  Clearly,
\[
 \PP(e(G[A])=1)=Np q^{N-1}.
\]
There are $\binom nm$ choices for $A$. Using
\eqref{eq:critical-mu}, we have 
\begin{equation}\label{eq:EY}
 \EE Y_m
 =\binom nmNp q^{N-1}
 =\frac pqN\mu_m(n)
 \sim\frac{p}{2q}m\longrightarrow\infty.
\end{equation}

\smallskip

For each $A\in\binom{[n]}m$, let $I_A$ denote the indicator of the event
$e(G[A])=1$.
Then $Y_m=\sum_A I_A$.   Let $A$ and $B$ be two distinct $m$-sets with
$|A\cap B|=j$, and set $L=\binom j2$.  There are exactly two
compatible possibilities for $I_A=I_B=1$:
\begin{enumerate}[label=\textup{(\alph*)}]
 \item their unique edge is the same edge inside $A\cap B$, giving $L$ choices;
 \item their unique edges are distinct and lie outside $A\cap B$, giving $(N-L)^2$ choices.
\end{enumerate}
Thus,
\begin{align}
 &\PP(e(G[A])=e(G[B])=1)\notag\\
 &\quad=Lp q^{2N-L-1}+(N-L)^2p^2q^{2N-L-2}.
\end{align}
After division by $(Np q^{N-1})^2$, the first and second terms
become
$
 \frac{qL}{pN^2}b^L
 $ and
 $\left(1-\frac LN\right)^2b^L$, respectively.
Since $0\le L\le N$, their sum is at most
$C_p b^L$ with $C_p=1+q/p$. Hence,
\begin{align}
 &\PP(e(G[A])=e(G[B])=1)
 \leq C_p b^{L}(Np q^{N-1})^2.
\end{align}

\smallskip

When $|A\cap B|=0,1$, then events $I_A=1$ and $I_B=1$ are independent.  So, their covariance is zero.  For $2\le j\le m-1$ and a fixed $m$-set $A$, the number of
sets $B$ with $|A\cap B|=j$ is
$\binom mj\binom{n-m}{m-j}$.  Write the variance of the indicator sum as
\[
 \operatorname{Var}Y_m
 =\sum_A\operatorname{Var}I_A
  +\sum_{\substack{A,B\in\binom{[n]}m\\A\ne B}}
    \operatorname{Cov}(I_A,I_B),
\]
where the second sum is over ordered pairs.  Since $I_A^2=I_A$,
\[
 \operatorname{Var}I_A
 =\EE I_A-(\EE I_A)^2\le\EE I_A.
\]
Thus the diagonal terms contribute at most
$\sum_A\EE I_A=\EE Y_m$.  For the remaining terms we use
$$\operatorname{Cov}(I_A,I_B)\le\EE(I_AI_B)=\PP(e(G[A])=e(G[B])=1)
 \leq C_p b^{L}(Np q^{N-1})^2.$$
   Dividing the resulting ordered-pair sum by
$(\EE Y_m)^2=\binom nm^2(Np q^{N-1})^2$ gives
\begin{equation}\label{eq:variance-overlap-sum}
 \frac{\operatorname{Var}Y_m}{(\EE Y_m)^2}
 \le\frac1{\EE Y_m}
 +C_p\sum_{j=2}^{m-1}
 \frac{\binom mj\binom{n-m}{m-j}}{\binom nm}
 b^{\binom j2}.
\end{equation}

\smallskip

Since $m/n\to0$, we may apply
\cref{lem:elementary-estimates}(iii) with $N=n$ and $r=m$.  This gives
\[
 \log\binom nm
 =m\log\frac nm+m
  +O\!\left(\frac{m^2}{n}+\log m\right).
\]
Since
$
 \frac{m^2/n}{m}=\frac mn\longrightarrow0
 $
 and $
 \frac{\log m}{m}\longrightarrow0,
$
the entire error is $o(m)$.  Together with
\eqref{eq:critical-mu}, this gives
\[
 -\log m+o(1)
 =\log\mu_m(n)
 =m\log\frac nm+m-\frac{m(m-1)}2\log b+o(m),
\]
After division by $m$,
\[
 \log\frac nm
 =\frac{m-1}{2}\log b-1+o(1).
\]
Then
\begin{align*}
 \log n
 &=\frac{m-1}{2}\log b-1+\log m+o(1)\\
 &=\frac m2\log b+o(m).
\end{align*}
  It follows that
\begin{equation}\label{eq:m-logn-relation}
 m=(2+o(1))\log_b n,\qquad b^m=n^{2+o(1)}.
\end{equation}
Let
\[
 T_j=\frac{\binom mj\binom{n-m}{m-j}}{\binom nm}
 b^{\binom j2}
\] for $2\leq j\leq m-1$.
We verify explicitly that \cref{lem:overlap-sum} applies with $t=m$.
The first relation in \eqref{eq:m-logn-relation} gives
$m=(2+o(1))\log_b n$.  Moreover, by \eqref{eq:critical-mu},
\[
 \log^+\!\bigl(\mu_m(n)^{-1}\bigr)
 =\log m+o(1)=o(\log n),
\]
where the last equality follows from $\log n=(\tfrac12\log b+o(1))m$.
Thus all the hypotheses of \cref{lem:overlap-sum} are satisfied, and
$\sum_{j=2}^{m-1}T_j=o(1)$. Combining this with \eqref{eq:EY} and
\eqref{eq:variance-overlap-sum}, we obtain
\[
 \frac{\operatorname{Var}Y_m}{(\EE Y_m)^2}=o(1).
\]
 Using Chebyshev's inequality, we have
\[
 \PP(Y_m=0)
 \le\PP\bigl(|Y_m-\EE Y_m|\ge\EE Y_m\bigr)
 \le\frac{\operatorname{Var}Y_m}{(\EE Y_m)^2}=o(1).
\]
This completes the proof.
\end{proof}

Now we are ready to prove \cref{thm:exact-instability}.

\begin{proof}[Proof of \cref{thm:exact-instability}]
\smallskip
Set $q=1-p$. For each sufficiently large $m$, let $n_m$ be the least integer $n$ such that
\begin{equation}\label{eq:nm-definition}
 \mu_m(n)=\binom nmq^{\binom m2}\ge\frac1m.
\end{equation}
This integer exists.  Indeed, for fixed $m$, the function
$n\mapsto\binom nm$ is strictly increasing for integers $n\ge m$, and it
tends to infinity with $n$.  Hence $\mu_m(n)$ is strictly increasing and
tends to infinity.  On the other hand,
$\mu_m(m)=q^{\binom m2}<1/m$ for all sufficiently large $m$, so the level
$1/m$ is crossed at a finite integer $n_m>m$.  In particular,
$n_m\ge m$ and hence $n_m\to\infty$.

We next prove the stronger fact $n_m/m\to\infty$. In fact, for any fixed $K>1$,
using
\cref{lem:elementary-estimates}(i), we have 
\[
 \mu_m(\lfloor Km\rfloor)
 \le (eK)^m q^{\binom m2}=o(1/m).
\]
Thus, $n_m>Km$ for all sufficiently large
$m$.  Since $K$ was arbitrary, $n_m/m\to\infty$.

For $n=n_m$, noting $m/n\to0$, we have
\[
 \frac{\mu_m(n)}{\mu_m(n-1)}=\frac{n}{n-m}=1+o(1).
\]
 By the minimality of $n_m$,
\[
 \mu_m(n_m-1)<\frac1m\le\mu_m(n_m).
\]
Multiplying the strict inequality by
$\mu_m(n_m)/\mu_m(n_m-1)=n_m/(n_m-m)$ gives
\[
 \frac1m\le\mu_m(n_m)
 <\frac1m\frac{n_m}{n_m-m}
 =\frac{1+o(1)}m.
\]
Therefore
\begin{equation}\label{eq:mu-asymptotic}
 \mu_m(n_m)\sim\frac1m.
\end{equation}

\smallskip

Let $X_m$ count the number of independent $m$-sets in $G(n_m,p)$. Then $\EE X_m=\mu_m(n_m)$.  Using Markov's inequality and \eqref{eq:mu-asymptotic}, we have
\[
 \PP(X_m>0)=\PP(X_m\geq1)\le\EE X_m=\mu_m(n_m)=o(1).
\]
This implies that 
\begin{equation}\label{eq:alpha-upper-m-1}
 \alpha(G(n_m,p))\le m-1
\end{equation}
with high probability. On the other hand, \eqref{eq:mu-asymptotic} verifies the hypothesis of
\cref{lem:one-edge-second-moment}.  Hence, with high probability, there is an
$m$-set $W$ inducing exactly one edge, say $xy$.  The set
$W\setminus\{x\}$ is independent. Thus $\alpha(G(n_m,p))\geq m-1$.  Together with
\eqref{eq:alpha-upper-m-1}, we have
\[
 \alpha(G(n_m,p))= m-1.
\]
Note that
$
 c(G[W])=m-1=\alpha(G(n_m,p)).
$
By \cref{lem:connected-noncomplete}, the graph $G(n_m,p)$ is connected and
non-complete with probability $1-o(1)$. Let $S$ be the complement of $W$. Noting that $c(G-S)=c(G[W])=m-1>1$, we have 
\[
 \tau(G)
 \le\frac{|S|}{c(G-S)}
 =\frac{n_m-m}{m-1}
 =\frac{n_m-(m-1)}{m-1}-\frac1{m-1}
 =\frac{n_m-\alpha(G)}{\alpha(G)}-\frac{1}{\alpha(G)}.
\]
All events above have probability $1-o(1)$; their intersection does as
well.  This completes the proof.
\end{proof}

\medskip

\medskip

\section*{Data availability}

No data were used for the research described in this article.

\section*{Competing Interests}
The authors declare that they have no known competing financial interests or
personal relationships that could have appeared to influence the work reported in
the paper.

\section*{Declaration on the Use of AI}
We used the ChatGPT for  proofreading, calculation checking, grammar checking, and language polishing.  The authors take full
responsibility for the correctness and originality of the paper.

\end{document}